\documentclass[reqno,12pt]{amsart}

\usepackage{amsmath,amsthm,amssymb}
\newtheorem{theorem}{Theorem}[section]
\newtheorem{lemma}[theorem]{Lemma}
\newtheorem{corollary}[theorem]{Corollary}
\newtheorem{remark}[theorem]{Remark}

\theoremstyle{definition}
\newtheorem{definition}[theorem]{Definition}

\numberwithin{equation}{section}

\begin{document}

\title[Besov classes on finite- and infinite-dimensional spaces]
{Besov classes on finite- and infinite-dimensional spaces and embedding theorems}

\author[Egor~D.~Kosov]{Egor~D.~Kosov}
\address{}
\curraddr{}
\email{}
\thanks{The author is a Young
Russian Mathematics award winner and would like to thank its sponsors and jury.\\
This research was supported by the Russian Science Foundation Grant 17-11-01058
at Lo\-mo\-no\-sov
Moscow State University.
}

\maketitle

\begin{abstract}
We give a new description of classical Besov spaces in terms of a new modulus of continuity.
Then a similar approach is used to introduce Besov classes on an infinite-dimensional space endowed with a Gaussian measure.
\end{abstract}

\noindent
Keywords: Besov class, fractional Sobolev class, Ornstein--Uhlenbeck semigroup, embedding theorem

\noindent
MSC: primary  46E35, secondary 28C20, 46G12

\section{Introduction}
In this work we continue the study of Nikolskii--Besov classes started
in  \cite{BKP}, where an equivalent
description of these classes was presented
characterizing the inclusion of a function to the Nikolskii--Besov class in terms
of action on test functions in
the spirit of the classical definitions of Sobolev classes and the class
of functions of bounded variation. Namely,
a function $f\in L^p(\mathbb{R}^n)$ belongs to the Nikolskii--Besov class $B^\alpha_{p, \infty}(\mathbb{R}^n)$ with
$0<\alpha<1$ if and only if
there is a constant $C$ such that
\begin{equation}\label{eq1}
\int_{\mathbb{R}^n} \mathrm{div}\Phi(x) f(x)\, dx \le C\|\Phi\|^{\alpha}_q\|\mathrm{div}\Phi\|^{1-\alpha}_q
\end{equation}
for each vector field
$\Phi$ of class $C_0^\infty(\mathbb{R}^n, \mathbb{R}^n)$, where $q= p/(p-1)$.
If we take $\alpha=1$ and $p=1$, we obtain the classical definition of a function
of bounded variation.
This new characterization has already found some applications in
the study of the distributions of polynomials on spaces with Gaussian
(and general log-concave) measures
(see \cite{Kos}, \cite{BKZ}, and also \cite{B16}).

In the present paper, we give a similar
equivalent characterization
 for  general Besov spaces $B^\alpha_{p,\theta}(\mathbb{R}^n)$.
We recall that the Besov space $B^\alpha_{p,\theta}(\mathbb{R}^n)$
with parameters
$\alpha\in(0,1)$, $p\in[1, \infty)$, $\theta\in[1, \infty]$
consists of all functions $f\in L^p(\mathbb{R}^n)$ such that
the quantity
$$
\biggl(\int_{\mathbb{R}^n}\bigl[|h|^{-\alpha}\|f_h-f\|_p\bigr]^\theta |h|^{-n}dh\biggr)^{1/\theta}
$$
is finite,
where $f_h(x) := f(x-h)$ (see \cite{BIN}, \cite{Nikol77}, \cite{Trieb},
and \cite{Stein}).
However, for further purposes, it is more convenient to use
another equivalent definition in terms
of the $L^p$-modulus of continuity.
Recall that the $L^p$-modulus of continuity of a function
$f\in L^p(\mathbb{R}^n)$ is defined by the equality
$$
\omega_p(f, \varepsilon):= \sup_{|h|\le \varepsilon}\|f_h - f\|_p.
$$
Note that the function $\omega_p(f, \cdot)$ is nondecreasing and subadditive,
which means that
$$
\omega_p(f, \varepsilon_1 + \varepsilon_2)\le
\omega_p(f, \varepsilon_1) + \omega_p(f, \varepsilon_2),
\quad  \varepsilon_1, \varepsilon_2>0.
$$
A function $f\in L^p(\mathbb{R}^n)$ belongs to the class $B^\alpha_{p,\theta}(\mathbb{R}^n)$
if and only if the quantity
$$
\|f\|_{\alpha, p, \theta}:=\biggl(\int_0^{+\infty}\bigl[s^{-\alpha}\omega_p(f,s)\bigr]^\theta s^{-1}ds\biggr)^{1/\theta}
$$
is finite.
We define the Besov norm of a function $f$ by the equality
$$
\|f\|_{B^\alpha_{p, \theta}(\mathbb{R}^n)} := \|f\|_p + \|f\|_{\alpha, p, \theta}.
$$

Our equivalent characterization
of Besov spaces is based on a new modulus of continuity
which is equivalent to $\omega_p(f, \cdot)$ and
provides the known characterization (\ref{eq1})
in the case of $\theta = \infty$.
For a function $f\in L^p(\mathbb{R}^n)$ we introduce
$$
\sigma_p(f, \varepsilon):=\sup\Bigl\{ \int_{\mathbb{R}^n} {\rm div}\Phi(x) f(x) dx,\,
\Phi\in C^\infty_0(\mathbb{R}^n, \mathbb{R}^n),
\|{\rm div}\Phi\|_{\frac{p}{p-1}}\le1, \|\Phi\|_{\frac{p}{p-1}}\le\varepsilon\Bigr\}.
$$
The first main result of the present paper
asserts the equivalence of $\omega_p(f, \cdot)$ and $\sigma_p(f, \cdot)$:
for any function $f\in L^p(\mathbb{R}^n)$, one has
$$
2^{-1}\omega_p(f, 2\varepsilon)\le \sigma_p(f, \varepsilon)
\le 6n\, \omega_p(f, \varepsilon).
$$
Actually,
the function $\sigma_p(f, \cdot)$ has
appeared implicitly
in the new definition of Nikolskii--Besov spaces
formulated above, since
condition (\ref{eq1}) can be reformulated in the following way:
$$
\sup_{s\ge0}s^{-\alpha}\sigma_p(f, s)<\infty.
$$
So, this is the desired modulus of continuity.
The above equivalence also shows
that a function $f\in L^p(\mathbb{R}^n)$ belongs to the Besov
space $B^\alpha_{p,\theta}(\mathbb{R}^n)$ if and only if
\begin{equation}\label{besov}
\Bigl(\int_0^\infty \bigl[s^{-\alpha}\sigma_p(f, s)\bigr]^\theta s^{-1}ds\Bigr)^{1/\theta}
<\infty.
\end{equation}

To illustrate how our approach to the fractional
smoothness in terms of the modulus of continuity $\sigma_p(f, \cdot)$
is related to the already known results,
in Section \ref{sect3} we propose the new proof of the
classical Ulyanov-type embedding theorems
by means of the function $\sigma_p(f, \cdot)$.
We recall that in his seminal works \cite{UlIzv}, \cite{UlMSb} P.L.~Ulyanov
obtained the following embedding theorem.

{\bf Theorem.}
For a function $f\in L^1[0,1]$ set
$$
w_1(f,\varepsilon):= \sup_{0\le h\le \varepsilon} \int_0^{1-h}|f(t+h) - f(t)|\, dt
$$
Then for any nondecreasing function $U\colon[0, \infty)\to[0, \infty)$
 the following implications hold:
$$
{\rm(i)}\, \sum_{n=1}^\infty [U(n+1) - U(n)]w_1(f, 1/n)<\infty \Rightarrow \int_0^1 |f(t)|U(|f(t)|)\, dt<\infty;
$$
$$
{\rm (ii)}\, \sum_{n=1}^\infty n^{-2} U(36 nw_1(f, 1/n))<\infty \Rightarrow \int_0^1 U(|f(t)|)\, dt<\infty.
$$
Actually in the same works  embedding theorems into  $L^r$-spaces were obtained,
but here we discuss only the stated results as examples of embedding theorems.
The multidimensional case was considered in papers \cite{Kol1}, \cite{Kol2} and \cite{Kol3},
where the author obtained
necessary and sufficient conditions for such type of embeddings.
The main method used by P.L. Ulyanov himself and by other researchers
in subsequent investigations  of such embedding theorems
is based on the so-called equimeasurable rearrangements of functions
(see \cite{Kol3} for a discussion of the method).
However, in Section \ref{sect3} we employ another approach,
based on the properties of the function $\sigma_p(f, \cdot)$,
and obtain similar simple
sufficient conditions for embeddings into the classes $LU(L)$ and $U(L)$.
Actually, our conditions are a kind of
integral form of Ulyanov's conditions
stated above and are similar to the multidimensional results
\cite[Theorem 1]{Kol1} and \cite[Corollary 4.2]{Kol3},
which are slightly weaker
than necessary and sufficient conditions \cite[Theorem 5]{Kol2} and \cite[Theorem 4.4]{Kol3}.
The main idea in Section \ref{sect3} is to estimate the integral
$$
\int_A |f(x)|^p dx
$$
of a function
$f\in L^p(\mathbb{R}^n)$ over a Borel set $A$ in terms of Lebesgue
measure of this set $A$. Substituting  $A=\{|f|\ge s\}$
we can estimate the behavior of the function $f$
on sets of large values, which
is already sufficient to prove embedding theorems we are interested in.
The definition of the function $\sigma_p(f, \cdot)$ states that
$$
\int_{\mathbb{R}^n} {\rm div}\Phi(x)f(x)\, dx\le \sigma_p(f, r)
$$
for smooth vector fields $\Phi$ with $\|\Phi\|_{\frac{p}{p-1}}\le r$ and $\|{\rm div}\Phi\|_{\frac{p}{p-1}}\le1$.
Taking $\Phi = \nabla\varphi$,
solving the Poisson equation ${\rm div}\nabla\varphi = \Delta\varphi = u$,
estimating $\nabla\varphi$ in terms of $u$, and taking the supremum over
functions $u$ with $\|u\|_{L^q(A)}=1$ we obtain the necessary bound.

Finally, in Section \ref{sect4}, we proceed to Besov classes on locally convex spaces
endowed with  centered Gaussian measures.
In paper \cite{BKP}, Nikolskii--Besov classes on a Gaussian space were introduced
by means of relation (\ref{eq1}) as the definition, where in place of the
divergence operator on  $\mathbb{R}^n$
 the Gaussian divergence operator ${\rm div}_\gamma$ was used.
If we consider the standard Gaussian measure $\gamma_n$
on $\mathbb{R}^n$, which is the measure with density
$(2\pi)^{-n/2}\exp(-|x|^2/2)$,
then
$${\rm div}_{\gamma_n}\Phi = \sum_{i=1}^n (\partial_i\Phi_i - x_i\Phi_i)
= {\rm div}\Phi - \langle x, \Phi\rangle.
$$
In this paper we propose a similar approach (see Definitions \ref{D4.1} and \ref{D4.2})
to general Besov classes $B^\alpha_{p,\theta}(\gamma)$
with respect to a Gaussian measure $\gamma$.
The first main result of Section \ref{sect4} (presented in Theorem \ref{t4.1})
provides an equivalent characterization of the introduced Besov classes
in terms of ``shifts'' on the Gaussian space,
which is similar in a sense to the classical definition
of Besov spaces on  $\mathbb{R}^n$.
Namely,  the function $f\in L^p(\gamma)$ with $p>1$
belongs to the Besov class $B^\alpha_{p,\theta}(\gamma)$
if and only if the quantity
$$
\Bigl(\int_0^\infty \Bigl[t^{-\alpha/2}\Bigl(\iint |f(e^{-t}x+ \sqrt{1- e^{-2t}}y) - f(x)|^p\,
\gamma(dx)\gamma(dy)\Bigr)^{1/p}\Bigr]^\theta t^{-1}dt\Bigr)^{1/\theta}
$$
is finite.
This theorem can be also viewed as an analog of Theorem 3.2 from \cite{AMP}.
The second main result of this section is the embedding theorem for Gaussian Besov classes.
We recall (see for example \cite{GM} and \cite{LedLS})
that for an arbitrary function $f$ from the Gaussian Sobolev space $W^{2,1}(\gamma)$
the following logarithmic Sobolev inequality holds:
$$
\int f^2\ln(|f|\|f\|_2^{-1})d\gamma\le \int|\nabla f|^2d\gamma.
$$
For the Sobolev class $W^{1,1}(\gamma)$ there is also an embedding theorem of logarithmic type.
Namely, the space $W^{1,1}(\gamma)$ is continuously embedded into the Orlicz space
 $L\log L^{1/2}$,
which is defined by the condition
 $$\int |f|[\ln(1+|f|)]^{1/2}d\gamma<\infty$$
(see \cite{LedIGA} and \cite{FH} for the case of functions of bounded variation).
Both  results mean that a smoothness of a function provides
some higher order of integrability.
One may wonder whether  this effect remains in force
for the Besov smoothness condition introduced in the present paper.
Theorem \ref{t4.2} is aimed to answer this question. It
asserts that, for any $\alpha\in(0,1)$, $\beta\in(0,\alpha)$,
$p\in (1,\infty)$, and $\theta\in[1, \infty]$, there is a constant
$C = C(p, \theta, \alpha, \beta)$ such that
for all functions $f\in B^\alpha_{p, \theta}(\gamma)$ one has
$$
\Bigl(\int|f|^p\bigl|\ln(|f|\|f\|_p^{-1})\bigr|^{p\beta/2} \, d\gamma\Bigr)^{1/p}
\le
C\|f\|_{B^\alpha_{p, \theta}(\gamma)}.
$$
The main idea of the proof of this result is in spirit of the semigroup approach
to the isoperimetric inequality on the Gaussian space proposed by M. Ledoux
in \cite{Led} and \cite{LedIGA}.
Similarly to the cited works, we use the short time behavior of
the Ornstein-Uhlenbeck semigroup on functions from the Besov class
and the hypercontractivity property of the Ornstein--Uhlenbeck semigroup.
At the end of the paper we provide an estimate of the best approximation
of a function from $L^2(\gamma)$ by Hermite polynomials in terms of the introduced
Gaussian modulus of continuity.

Throughout the paper we assume that $\alpha$ is a fixed number from $(0,1]$.
Given $p\in[1,\infty]$,  we denote by $q$ the dual number
such that $1/p+1/q=1$.
The $L^p$-norm of a function $f$ with respect to a measure $\mu$ is defined
as usual by
$$
\|f\|_p := \|f\|_{L^p(\mu)} = \Bigl(\int |f|^p\, d\mu\Bigr)^{1/p},\, p\in[1,\infty),
$$
and the limiting case of $p=\infty$ is treated also as usual.
In Sections \ref{sect1} and \ref{sect3} the measure $\mu$ will be
the standard Lebesgue measure on $\mathbb{R}^n$,
but in Section \ref{sect4} the measure $\mu$ will be a centered
Gaussian measure on a locally convex space.
We denote the space
of all infinitely differential functions with compact support
on $\mathbb{R}^n$ by $C_0^\infty(\mathbb{R}^n)$ and
the space of all bounded infinitely
differential functions with bounded derivatives of every order
is denoted by $C_b^\infty(\mathbb{R}^n)$.

\section{Besov classes on $\mathbb{R}^n$}\label{sect1}

This section is devoted to obtaining a new characterization
of Besov classes on $\mathbb{R}^n$
in terms of the moduli of continuity $\sigma_p(f, \cdot)$ and $\widetilde{\sigma}_p(f, \cdot)$.

Let $|\cdot|$ denote the standard Euclidean norm on $\mathbb{R}^n$
generated by the standard Euclidean inner product $\langle\cdot, \cdot\rangle$.
Let $\lambda^n$ be the standard Lebesgue measure on $\mathbb{R}^n$.
We also need the heat semigroup
$P_t$ on $\mathbb{R}^n$, which is defined by the equality
$$
P_tf(x) := (2\pi t)^{-n/2}\int_{\mathbb{R}^n}f(y)\exp\biggl(-\frac {|x-y|^2}{2t}\biggr)\, dy, \quad f\in L^1(\mathbb{R}^n).
$$

We start with the following key definitions (recall that $q=p/(p-1)$).

\begin{definition}\label{D1.1} Let $f\in L^p(\mathbb{R}^n)$.
Set
$$
\sigma_p(f, \varepsilon):=\sup\Bigl\{ \int_{\mathbb{R}^n} {\rm div}\Phi(x) f(x)\, dx:\,
\Phi\in C^\infty_0(\mathbb{R}^n, \mathbb{R}^n), \|{\rm div}\Phi\|_q\le1, \|\Phi\|_q\le\varepsilon\Bigr\}.
$$
\end{definition}

\begin{definition}\label{D1.3} Let $f\in L^p(\mathbb{R}^n)$.
Set
$$
\widetilde{\sigma}_p(f, \varepsilon):=\sup\Bigl\{ \int_{\mathbb{R}^n} \partial_e\varphi(x) f(x)\, dx:\,
\varphi\in C^\infty_0(\mathbb{R}^n), \|\partial_e\varphi\|_q\le1, \|\varphi\|_q\le\varepsilon\Bigr\}.
$$
\end{definition}

We now obtain several properties of the introduced functions.

\begin{lemma}\label{lem1.1}
For any function $f\in L^p(\gamma)$,
the functions
$\sigma_p(f, \cdot)$ and $\widetilde{\sigma}_p(f, \cdot)$
are nondecreasing, subadditive, concave and continuous on $(0, +\infty)$.
\end{lemma}
\begin{proof}
We consider only the function $\sigma_p(f, \cdot)$, since for the second
one the proof is essentially the same.
It is readily seen that this function is indeed nondecreasing and subadditive.
We now check that it is concave. Let $a, b>0$, and $t\in(0, 1)$. Then
for an arbitrary pair of vector fields $\Phi_1, \Phi_2\in C^\infty_0(\mathbb{R}^n, \mathbb{R}^n)$
with $\|{\rm div}\Phi_1\|_q\le1, \|\Phi_1\|_q\le a$
and $\|{\rm div}\Phi_2\|_q\le1, \|\Phi_2\|_q\le b$
we have
$$
t\int_{\mathbb{R}^n} {\rm div}\Phi_1(x) f(x)\, dx + (1-t)\int_{\mathbb{R}^n} {\rm div}\Phi_2(x) f(x)\, dx
=
\int_{\mathbb{R}^n} {\rm div}[t\Phi_1(x) + (1-t)\Phi_2(x)]f(x)\, dx
$$
and
$\|{\rm div}[t\Phi_1 + (1-t)\Phi_2]\|_q\le1$, $\|t\Phi_1 + (1-t)\Phi_2\|_q\le ta + (1-t)b$.
Thus,
$$
t\sigma_p(f, a) + (1-t)\sigma_p(f, b)\le \sigma_p(f, ta+(1-t)b).
$$
The concavity implies the continuity.
\end{proof}

Let $\sigma$ be a concave, nondecreasing and nonnegative function on $(0, +\infty)$.
Let us  introduce the ``adjoint'' function
$$
\sigma^*(s):= s\sigma(s^{-1}).
$$
In particular, we can consider
$$
\sigma_p^*(f, s):= s\sigma_p(f, s^{-1}), \quad \widetilde{\sigma}_p^*(f, s):= s\widetilde{\sigma}_p(f, s^{-1}).
$$

\begin{lemma}\label{lem1.2}
Let $\sigma$ be a concave, nondecreasing and nonnegative function on $(0, +\infty)$.
Then the function $\sigma^*$
is also concave and nondecreasing on $(0, +\infty)$.
If, in addition, we assume that $\lim_{t\to0}t^{-1}\sigma(t)=\infty$, then
the function $\sigma^*$ is strictly monotone.
\end{lemma}

\begin{proof}
Let $s, t\in (0, +\infty)$ and $s>t$. Then $1/s= (t/s)1/t + (1-t/s)0$.
Due to the concavity of the function $\sigma$ we have
$$
\sigma((t/s)1/t + (1-t/s)\varepsilon)\ge (t/s)\sigma(1/t) + (1-t/s)\sigma(\varepsilon) \ge (t/s)\sigma(1/t).
$$
Due to the continuity of the function $\sigma$,
taking the limit as $\varepsilon\to 0$ in the above estimate, we have
$\sigma(1/s)\ge (t/s)\sigma(1/t)$ implying
$\sigma^*(s)\ge\sigma^*(t)$.

Let again $s>t>0$ and let $\varkappa\in(0,1)$.
We note that
$$
\frac{1}{\varkappa s + (1-\varkappa) t} = \frac{\varkappa s} {\varkappa s + (1-\varkappa) t} 1/s + \frac{(1-\varkappa) t}{\varkappa s + (1-\varkappa) t}1/t.
$$
Thus, by the concavity of the function $\sigma$ one has
$$
\sigma((\varkappa s + (1-\varkappa) t)^{-1})\ge \frac{\varkappa s} {\varkappa s + (1-\varkappa) t}\sigma(1/s) + \frac{(1-\varkappa) t}{\varkappa s + (1-\varkappa) t}\sigma(1/t)
$$
and $\sigma^*(\varkappa s + (1-\varkappa) t)\ge \varkappa\sigma^*(s) + (1-\varkappa)\sigma^*(t)$, i.e. $\sigma^*$ is concave.

Assume that there are two points $s>t$
such that $\sigma^*(s) = \sigma^*(t)$. Then, by the concavity, for any point $r>s$ one has
$$
\sigma^*(r) = \sigma^*\biggl(\frac{r-t}{s-t} s + \Bigl(1-\frac{r-t}{s-t}\Bigr)t\biggr)\le
\frac{r-t}{s-t} \sigma^*(s) + \Bigl(1-\frac{r-t}{s-t}\Bigr)\sigma^*(t) = \sigma^*(s).
$$
Thus,
$$
\lim_{r\to\infty}r\sigma(1/r)=\lim_{r\to\infty}\sigma^*(r)\le\sigma^*(s)
$$
which contradicts the condition $\lim_{t\to0}t^{-1}\sigma(t)=\infty$.
\end{proof}

\begin{corollary}\label{c1.1}
Let $f\in L^p(\mathbb{R}^n)$.
Then the functions $\sigma_p^*(f, \cdot)$
and
$\widetilde{\sigma}_p^*(f, \cdot)$
are concave and nondecreasing on $(0, +\infty)$.
If, in addition, we assume that $\lim_{t\to0}t^{-1}\sigma_p(f, t)=\infty$
(alternatively, $\lim_{t\to0}t^{-1}\widetilde{\sigma}_p(f, t)=\infty$), then
the function $\sigma_p^*(f, \cdot)$ ($\widetilde{\sigma}_p^*(f, \cdot)$, respectively)
is strictly monotone.
\end{corollary}

We now proceed to the main result of this section showing the equivalence
of $\sigma_p(f, \cdot)$, $\widetilde{\sigma}_p(f, \cdot)$,
and $\omega_p(f, \cdot)$.
We start with the following technical lemma
(the proof is similar to the proof of Theorem 3.4 from \cite{BKP}).

\begin{lemma}\label{t1.1}
For any function $f\in L^p(\mathbb{R}^n)$ one has
$$
2^{-1}\|f_{2h} - f\|_p\le \widetilde{\sigma}_p(f, |h|)\le \sigma_p(f, |h|)
\le (2\pi)^{-n/2} \int_{\mathbb{R}^n}
\|f_{|h|z}-f\|_p (1+|z|)e^{-\frac{|z|^2}{2}}\, dz.
$$
\end{lemma}
\begin{proof}
For every function $\varphi\in C_0^\infty(\mathbb{R}^n)$
and for every unit vector $e\in\mathbb{R}^n$
we can take $\Phi=e\phi$ and conclude that
$$
\widetilde{\sigma}_p(f, \varepsilon)\le \sigma_p(f, \varepsilon).
$$
Let now $e = |h|^{-1}h$. For an arbitrary function
$\varphi\in C_0^\infty(\mathbb{R}^n)$ with $\|\varphi\|_q \le1$
we can write
\begin{multline*}
\int_{\mathbb{R}^n} \varphi(x)(f_h(x)-f(x))\, dx
=
\int_{\mathbb{R}^n} [\varphi(x+h)-\varphi(x)] f(x)\, dx
\\
=
\int_{\mathbb{R}^n} \int_0^{|h|}\partial_e\varphi(x+se)\, ds f(x)\, dx.
\end{multline*}
For the function
$$
\psi(x)=\int_0^{|h|}\varphi(x+se)\, ds \in C_0^\infty(\mathbb{R}^n)
$$
we have $\|\psi\|_q \le |h|\|\varphi\|_q\le |h|$
and $\|\partial_e\psi\|_q\le 2\|\varphi\|_q\le 2$,
since
$$
|\partial_e\psi(x)|=\biggl|\int_0^{|h|}\partial_e\varphi(x+se)\, ds\biggr|
=|\varphi(x+h)-\varphi(x)|.
$$
Thus,
$$
\int_{\mathbb{R}^n} \varphi(x)(f_h(x)-f(x))dx=\int_{\mathbb{R}^n}\partial_e\psi(x)f(x)\, dx\le
2\widetilde{\sigma}_p(f, |h|/2).
$$
Taking the supremum over functions $\varphi$,
we get the estimate $\|f_h - f\|_p\le2\widetilde{\sigma}_p(f, |h|/2)$.

Finally,
for every smooth
vector field $\Phi\in C_0^\infty(\mathbb{R}^n, \mathbb{R}^n)$
we can write
\begin{equation}\label{eq2}
\int_{\mathbb{R}^n} \mathrm{div}\Phi(x) f(x)\, dx =
\int_{\mathbb{R}^n} \mathrm{div}\Phi(x) (f(x) - P_tf(x))\, dx + \int_{\mathbb{R}^n} \mathrm{div}\Phi(x) P_tf(x)\, dx.
\end{equation}
We note that
$$
\|f-P_tf\|_p\le
(2\pi)^{-n/2} \int_{\mathbb{R}^n}
\|f_{\sqrt{t}z}-f\|_p e^{-\frac{|z|^2}{2}}\, dz.
$$
Thus, for the first term in equality (\ref{eq2})
we have
$$
\int_{\mathbb{R}^n} \mathrm{div}\Phi(x) (f(x) - P_tf(x))\, dx \le \|\mathrm{div}\Phi\|_q\|f - P_tf\|_p\le
\|\mathrm{div}\Phi\|_q(2\pi)^{-n/2} \int_{\mathbb{R}^n}
\|f_{\sqrt{t}z}-f\|_p e^{-\frac{|z|^2}{2}}\, dz.
$$
Integrating by parts in the second term of equality (\ref{eq2}), we have
\begin{multline*}
\int_{\mathbb{R}^n} \mathrm{div}\Phi(x) P_tf(x)\, dx = - \int_{\mathbb{R}^n}  \langle \Phi(x), \nabla P_tf(x)\rangle\, dx
\\
= t^{-1/2}\int_{\mathbb{R}^n}  \int_{\mathbb{R}^n}  \langle \Phi(x), (x-y)t^{-1/2}\rangle f(y)(2\pi t)^{-n/2}e^{-\frac{|x-y|^2}{2t}}\, dy\, dx
\\
= t^{-1/2}\int_{\mathbb{R}^n}  \int_{\mathbb{R}^n}  \langle \Phi(x), z\rangle f(x - \sqrt{t}z)(2\pi)^{-n/2}e^{-\frac{|z|^2}{2}}\, dz\, dx.
\end{multline*}
Since
$$\int f(x) \int \langle \Phi(x), z\rangle e^{-\frac{|z|^2}{2}}\,dz\, dx = 0,$$
 the above expression is equal to
\begin{multline*}
t^{-1/2}\int_{\mathbb{R}^n}  (2\pi)^{-n/2}e^{-\frac{|z|^2}{2}} \int_{\mathbb{R}^n}  \langle \Phi(x), z\rangle (f(x - \sqrt{t} z) - f(x))\, dx\, dz
\\
\le
t^{-1/2}\|\Phi\|_q(2\pi)^{-n/2}\int_{\mathbb{R}^n}|z|e^{-\frac{|z|^2}{2}} \|f_{\sqrt{t} z} - f\|_p\, dz
\end{multline*}
Thus, we have
\begin{multline*}
\int_{\mathbb{R}^n} \mathrm{div}\Phi(x) f(x)\, dx\le
\|\mathrm{div}\Phi\|_q(2\pi)^{-n/2} \int_{\mathbb{R}^n}
\|f_{\sqrt{t}z}-f\|_p e^{-\frac{|z|^2}{2}}\, dz
\\
+
t^{-1/2}\|\Phi\|_q(2\pi)^{-n/2}\int_{\mathbb{R}^n}|z| \|f_{\sqrt{t} z} - f\|_p e^{-\frac{|z|^2}{2}}\, dz.
\end{multline*}
Hence
$$
\sigma_p(f, \varepsilon)\le
(2\pi)^{-n/2} \int_{\mathbb{R}^n}
\|f_{\sqrt{t}z}-f\|_p e^{-\frac{|z|^2}{2}}\, dz
+
t^{-1/2}\varepsilon(2\pi)^{-n/2}\int_{\mathbb{R}^n}|z| \|f_{\sqrt{t} z} - f\|_p e^{-\frac{|z|^2}{2}}\, dz
$$
and taking $\sqrt{t} = \varepsilon$ we conclude that
$$
\sigma_p(f, \varepsilon)\le
(2\pi)^{-n/2} \int_{\mathbb{R}^n}(1+|z|)
\|f_{\varepsilon z}-f\|_p e^{-\frac{|z|^2}{2}}\, dz.
$$
The lemma is proved.
\end{proof}

As we have already mentioned in the introduction,
the function $\omega_p(f, \cdot)$ is nondecreasing and subadditive,
in particular,
\begin{equation}\label{est0}
\omega_p(f, \tau s)\le 2\tau \omega_p(f, s)
\end{equation}
for $\tau\ge1$ and $s>0$.
Indeed, let $k\in\mathbb{N}$ be a number such that $k\le \tau< k+1$.
Then
$$
\omega_p(f, \tau s)\le \omega_p(f, (k+1)s)\le (k+1)\omega_p(f, s)= k(1+1/k)\omega_p(f, s)\le2\tau\omega_p(f, s).
$$

Now we are ready to prove the aforementioned equivalence.

\begin{theorem}\label{T3.1}
For any function $f\in L^p(\mathbb{R}^n)$, we have
$$
2^{-1}\omega_p(f, 2\varepsilon)\le
\widetilde{\sigma}_p(f, \varepsilon)\le \sigma_p(f, \varepsilon)
\le 2(1 + \sqrt{n} + n)\omega_p(f, \varepsilon).
$$
\end{theorem}

\begin{proof}
The first two inequalities are straightforward corollaries of Lemma \ref{t1.1}.
For the last one, by the same lemma, we have
\begin{multline*}
\sigma_p(f, \varepsilon)\le
(2\pi)^{-n/2} \int_{\mathbb{R}^n}
\|f_{\varepsilon z}-f\|_p (1+|z|)e^{-\frac{|z|^2}{2}}\, dz
\le
(2\pi)^{-n/2} \int_{\mathbb{R}^n}
\omega_p(f, \varepsilon|z|) (1+|z|)e^{-\frac{|z|^2}{2}}\, dz
\\
=
(2\pi)^{-n/2} \int_{|z|\le1}
\omega_p(f, \varepsilon|z|) (1+|z|)e^{-\frac{|z|^2}{2}}\, dz
+
(2\pi)^{-n/2} \int_{|z|>1}
\omega_p(f, \varepsilon|z|) (1+|z|)e^{-\frac{|z|^2}{2}}\, dz.
\end{multline*}
The first integral above is estimated by
$$
\omega_p(f, \varepsilon)(2\pi)^{-n/2} \int_{|z|\le1}
(1+|z|)e^{-\frac{|z|^2}{2}}\, dz
\le
2\omega_p(f, \varepsilon)
$$
by the monotonicity of the function $\omega_p(f, \cdot)$.
The second integral, by estimate (\ref{est0}), is not greater than
$$
\omega_p(f, \varepsilon)(2\pi)^{-n/2} \int_{|z|>1}
2|z|(1+|z|)e^{-\frac{|z|^2}{2}}\, dz
\le
\omega_p(f, \varepsilon)(2\sqrt{n} + 2n).
$$
Combining these two estimates we get the announced bound.
\end{proof}

As a corollary of the above theorem we get an equivalent
characterization of Besov classes on $\mathbb{R}^n$.
Let us introduce the following notation.

\begin{definition}\label{D1.2}
Let $f\in L^p(\mathbb{R}^n)$, $p\in[1,\infty)$, $\theta\in[1,\infty]$, and $\alpha\in(0,1)$.
Set
$$
V^{p, \theta, \alpha}(f) = \Bigl(\int_0^\infty \bigl[s^{-\alpha}\sigma_p(f, s)\bigr]^\theta s^{-1}\, ds\Bigr)^{1/\theta}.
$$
\end{definition}

\begin{definition}\label{D1.4}
Let $f\in L^p(\mathbb{R}^n)$, $p\in[1,\infty)$, $\theta\in[1,\infty]$, and $\alpha\in(0,1)$.
Set
$$
\widetilde{V}^{p, \theta, \alpha}(f) =
\Bigl(\int_0^\infty \bigl[s^{-\alpha}\widetilde{\sigma}_p(f, s)\bigr]^\theta s^{-1}\, ds\Bigr)^{1/\theta}.
$$
\end{definition}

\begin{corollary}\label{T1.2}
For any function $f\in L^p(\mathbb{R}^n)$, the following conditions are equivalent:

$\rm(i)$ $f\in B^\alpha_{p, \theta}(\mathbb{R}^n)$;

$\rm(ii)$ $V^{p,\theta, \alpha}(f)<\infty$;

$\rm(iii)$ $\widetilde{V}^{p, \theta, \alpha}(f)<\infty$.

Moreover,
$$
2^{\alpha-1}\|f\|_{\alpha, p, \theta}\le \widetilde{V}^{p, \theta, \alpha}(f)\le V^{p,\theta, \alpha}(f)
\le 2(1 + \sqrt{n} + n)\|f\|_{\alpha, p, \theta}.
$$
\end{corollary}

\section{Ulyanov embedding theorems}\label{sect3}
In this section we study embedding theorems
by means of the obtained properties of the function $\sigma_p(f, \cdot)$.

Our first goal is to estimate the measure of the set $\{|f|\ge t\}$.
To provide such an estimate we  need the following lemma.

\begin{lemma}\label{lem2.1}
Let $f\in L^p(\mathbb{R}^n)$, $p\in[1,n)$ or $n=p=1$,
and let $u\in C_0^\infty(\mathbb{R}^n)$ be a function such that
$\|u\|_q\le 1$ (recall that $q=p/(p-1)$). Then
$$
\int u(x) f(x) dx
\le C(n, p)\sigma_p\bigl(f,\|u\|_1^{p/n}\bigr),
$$
where $C(n, p) = 1+\nu_n^{-1/p}(n/p-1)^{1/p-1}$,
 $\nu_n$ is the surface area of the unit sphere in $\mathbb{R}^n$,
and $C(1,1) = 1$.
\end{lemma}
\begin{proof}
By approximation, for an arbitrary vector
field $\Phi\in C^\infty(\mathbb{R}^n)$ with $\|\Phi\|_q\le \varepsilon$,
$\|{\rm div}\Phi\|_q\le1$
one has
$$
\int {\rm div}\Phi(x)f(x) dx \le \sigma_p(f, \varepsilon).
$$
Assume first that $n>2$. Consider the function
$$
\varphi(x) = - (n-2)^{-1}\nu_n^{-1}\int_{\mathbb{R}^n} |x-y|^{-n+2}u(y)\, dy.
$$
It is known that $\mathrm{div}\nabla\varphi = \Delta \varphi = u$.
Set
$$
K_1(x) = |x|^{-n+1}\mathrm{Ind}_{\{|x|<R\}}(x), \quad K_2(x) = |x|^{-n+1}\mathrm{Ind}_{\{|x|\ge R\}}(x).
$$
Let us estimate $\nabla\varphi$:
$$
|\nabla\varphi(x)| \le \nu_n^{-1} \int_{\mathbb{R}^n} |x-y|^{-n+1}|u(y)|\, dy =
\nu_n^{-1}\bigl(K_1*|u| (x)+K_2*|u| (x)\bigr).
$$
Thus,
\begin{multline*}
\|\nabla\varphi\|_q\le\nu_n^{-1}\bigl(\|K_1\|_1\|u\|_q + \|K_2\|_q\|u\|_1\bigr)\\
\le
\nu_n^{-1}\int_{\{|x|<R\}}|x|^{-n+1}\, dx + \nu_n^{-1}\|u\|_1\biggl(\int_{\{|x|\ge R\}}|x|^{-nq+q}\, dx\biggr)^{1/q}\\
=
R + \nu_n^{-1/p}\bigl((n-1)(q-1)-1\bigr)^{-1/q}R^{1-n/p}\|u\|_1 =
R + \nu_n^{-1/p}q^{-1/q}(n/p-1)^{1/p-1}R^{1-n/p}\|u\|_1\\
\le
R + \nu_n^{-1/p}(n/p-1)^{1/p-1}R^{1-n/p}\|u\|_1
\end{multline*}
Setting now $R = \|u\|_1^{p/n}$, we obtain
$$
\|\nabla\varphi\|_q \le \bigl(1+\nu_n^{-1/p}(n/p-1)^{1/p-1}\bigr)\|u\|_1^{p/n}.
$$
Thus,
\begin{multline*}
\int u(x) f(x) dx = \int \mathrm{div}\nabla\varphi(x) f(x) dx \le
\sigma_p\bigl(f, \bigl(1+\nu_n^{-1/p}(n/p-1)^{1/p-1}\bigr)\|u\|_1^{p/n}\bigr)
\\
\le \bigl(1+\nu_n^{-1/p}(n/p-1)^{1/p-1}\bigr)\sigma_p\bigl(f,\|u\|_1^{p/n}\bigr),
\end{multline*}
where we have used the monotonicity of the function $\sigma_p^*(f, \cdot)$ (see Corollary \ref{c1.1}).
We have obtained the announced estimate in the case $n>2$.
For $n=2$ we can take
$$
\varphi(x) = - (2\pi)^{-1}\int \ln|x-y|u(y)dy
$$
and argue as above.

Thus, only the case $n=1$ remains.
In that case we consider the function
$\varphi(x) = \int_{-\infty}^x u(t)dt$.
For this function we can write
$$
\int u(x)f(x) dx = \int \varphi'(x)f(x) dx
\le
\sigma_1(f, \|\varphi\|_\infty)
\le
\sigma_1(f, \|u\|_1).
$$
The lemma is proved.
\end{proof}

\begin{corollary}\label{c2.1}
Let $f\in L^p(\mathbb{R}^n)$, $p\in[1,n)$ or $n=p=1$,
then
$$
\biggl(\int_A |f(x)|^p\, dx\biggr)^{1/p} \le C(n, p) \sigma_p\bigl(f, \bigl(\lambda^n(A)\bigr)^{1/n}\bigr)
$$
for an arbitrary Borel set $A$ in $\mathbb{R}^n$ with $C(n, p) = 1+\nu_n^{-1/p}(n/p-1)^{1/p-1}$,
where $\nu_n$ is the surface area of the unit sphere in $\mathbb{R}^n$,
and $C(1,1) = 1$.
\end{corollary}

\begin{proof}
Assume first that $A$ is a bounded set, $A\subset B(0, R)$, where $B(0, R)$
is the ball of radius $R$ centered at the origin.
Consider a function $u\in L^q(A)$ with $\|u\|_{L^q(A)}\le1$.
There is a sequence of functions $u_m\in C_0^\infty(\mathbb{R}^n)$ such that
$\mathrm{supp}(u_m)\subset B(0, 2R)$, $\|u_m\|_q\le1$, and $u_m\to u$ in $L^q\bigl(B(0, 2R)\bigr)$
($\lambda^n$-a.e. in  case $p=1$), where
$u(x) = 0$ if $x\not\in A$. For example, such a sequence can
be constructed  by means of convolutions with
compactly supported smooth probability densities.
For each function $u_m$ by the previous lemma we have
$$
\int u_m(x) f(x)\, dx \le C(n, p)\sigma_p\bigl(f, \|u_m\|_1^{p/n}\bigr).
$$
Since $\|\cdot\|_1$ is a continuous function on the space $L^q\bigl(B(0, 2R)\bigr)$ for $p>1$
(or by  Lebesgue's dominated convergence theorem in case $p=1$),
the above estimate is also valid for the function $u$.
Thus, for every function $u\in L^q(A)$ with $\|u\|_{L^q(A)}\le1$ we have
$$
\int_A u(x) f(x)\, dx \le C(n, p) \sigma_p\bigl(f, \|u\|_1^{p/n}\bigr)
\le C(n, p) \sigma_p\bigl(f, \bigl(\lambda^n(A)\bigr)^{1/n}\bigr).
$$
Taking the supremum over all functions $u$ with $\|u\|_{L^q(A)}\le1$ we obtain the desired
estimate for bounded sets $A$. The case of an arbitrary set $A$
can be obtained by passing to the limit.
\end{proof}

We now proceed to embedding theorems, which we formulate in terms
of the function $\sigma_p(f, \cdot)$
instead of $\omega_p(f, \cdot)$, since these functions are equivalent by Theorem \ref{T3.1}.

\begin{theorem}\label{T3.2}
Let $f\in L^p(\mathbb{R}^n)$, where $p\in [1, n)$ or $n=p=1$.
Let $U\colon [0, \infty)\to[0,\infty)$ be a nondecreasing continuous
function.
Assume that there is a number $N>0$ such that
$$
\int_N^{+\infty}\bigl[\sigma_p\bigl(f, t^{-p/n}\bigr)\bigr]^p\, dU(t) <\infty,
$$
where the integral is understood in the Lebesgue--Stieltjes sense.
Then
$$
\int_{\mathbb{R}^n} |f(t)|^pU(|f(t)|\|f\|_p^{-1})\, dt <\infty.
$$
\end{theorem}

\begin{proof}
By Corollary \ref{c2.1} for the set $A_s:=\{|f|\ge \|f\|_ps\}$ we have
$$
\int_{A_s} |f(x)|^p\, dx \le C(n, p)^p\bigl[\sigma_p\bigl(f,(\lambda^n(A_s))^{1/n}\bigr)\bigr]^p
\le
C(n, p)^p \bigl[\sigma_p\bigl(f,s^{-p/n}\bigr)\bigr]^p,
$$
since $\lambda^n(A_s)\le s^{-p}$.
Integrating both sides of the above estimate from $N$ to $+\infty$ with respect to the locally bounded measure,
generated by the monotone function $U$, we get
$$
\int_N^{+\infty}\int_{A_s} |f(x)|^p\, dx\,  dU(s)\le
C(n, p)^p \int_N^{+\infty}\bigl[\sigma_p\bigl(f,s^{-p/n}\bigr)\bigr]^p\, dU(s)<\infty.
$$
We now deal with the left-hand side of the above estimate, which is equal to
\begin{multline*}
\int_{\mathbb{R}^n} I_{\{|f|\ge N\|f\|_p\}}|f(x)|^p\bigl[U(|f(x)|\|f\|_p^{-1}) - U(N)\bigr]\, dx
\\
=
\int_{\mathbb{R}^n} |f(x)|^pU(|f(x)|\|f\|_p^{-1})\, dx
-
\int_{\mathbb{R}^n} I_{\{|f|< N\|f\|_p\}}|f(x)|^pU(|f(x)|\|f\|_p^{-1})\, dx
\\
-
U(N)\int_{\mathbb{R}^n} I_{\{|f|\ge N\|f\|_p\}}|f(x)|^p\, dx
\ge
\int_{\mathbb{R}^n} |f(x)|^pU(|f(x)|\|f\|_p^{-1})\, dx
-
U(N)\|f\|_p^p.
\end{multline*}
Thus,
$$
\int_{\mathbb{R}^n} |f(x)|^pU(|f(x)|\|f\|_p^{-1})\, dx\le C(n, p)^p \int_N^{+\infty}\bigl[\sigma_p\bigl(f,s^{-p/n}\bigr)\bigr]^p\, dU(s) + U(N)\|f\|_p^p.
$$
The theorem is proved.
\end{proof}

We now proceed to the second theorem.
Let us introduce the following notation.
Let \mbox{$p\in[1, n)$} and let $f\in L^p(\mathbb{R}^n)$.
Consider the function
\begin{equation}\label{vstar}
v_p(f, t) = \sigma_p(f, t^{\frac{p}{n}}).
\end{equation}
Since the function $t\to t^{\frac{p}{n}}$ is concave,
the function $v_p(f, \cdot)$ is also concave and nondecreasing.
Moreover, $\lim_{t\to0} t^{-1}v_p(f, t) = \lim_{s\to0} s^{-n/p}\sigma_p(f, s)$.
We note that, by concavity,
$$
\sigma_p(f, s) = \sigma_p (f, s\cdot1 + (1-s)\cdot0)\ge s\sigma_p(f, 1).
$$
Thus, $\lim_{s\to0} s^{-n/p}\sigma(f, s) = \infty$ and
the function $v^*_p(f, t):= tv_p(f, t^{-1})$ is strictly increasing by Lemma \ref{lem1.2}.
To unify the notation, for the case $n=p=1$ we also use the symbol $v_p(f, \cdot)$,
which in this case coincides with $\sigma_p(f, \cdot)$.

\begin{lemma}\label{lem3.1}
Let $f\in L^p(\mathbb{R}^n)$, where $p\in[1,n)$ or $n=p=1$.
In the case $n=p=1$ we also assume that $\lim_{t\to0}t^{-1}\sigma_1(f, t)=\infty$.
Then
$$
\lambda^n(|f|\ge C(n,p)v_p^*(f, t))\le t^{-p}, \quad v_p^*(f, t) = t\sigma_p(f, t^{-\frac{p}{n}}),
$$
where $C(n, p) = 1+\nu_n^{-1/p}(n/p-1)^{1/p-1}$,
 $\nu_n$ is the surface area of the unit sphere in $\mathbb{R}^n$,
and $C(1,1) = 1$.
\end{lemma}

\begin{proof}
As we have already mentioned, the function $v^*$ is strictly monotone.
By Corollary \ref{c2.1}, for the set $A_t:= \{|f|\ge C(n ,p) v_p^*(f, t)\}$ we have
$$
C(n ,p)v_p^*(f, t)\lambda^n(A_t)^{1/p}\le
\biggl(\int_{A_t} |f(x)|^p\, dx\biggr)^{1/p}
\le C(n, p) \sigma_p\bigl(f, \bigl(\lambda^n(A_t)\bigr)^{1/n}\bigr).
$$
Thus,
$$
v_p^*(f, t)\le v_p^*\bigl(f, \bigl(\lambda^n(A_t)\bigr)^{-1/p}\bigr).
$$
By the strictly monotonicity of the function $v_p^*(f, \cdot)$ we have the estimate
$$
\lambda^n(A_t)\le t^{-p}
$$
which completes the proof.
\end{proof}

We now proceed to the theorem itself.

\begin{theorem}\label{T3.3}
Let $f\in L^p(\mathbb{R}^n)$, where $p\in[1,n)$ or $n=p=1$.
Let $U\colon [0, \infty)\to[0,\infty)$ be a strictly increasing
continuous function such that $U(0) = 0$, $\lim_{t\to\infty}U(t) = \infty$
and there are positive constants $a, r$ such that $U(t)\le at^p$ whenever $0<t<r$.
Assume that there is a number $N>0$ such that
$$
\int_N^{+\infty} t^{-1-p}U\bigl(C(n,p)t\sigma_p(f, t^{-p/n})\bigr)\, dt <\infty,
$$
where $C(n, p) = 1+\nu_n^{-1/p}(n/p-1)^{1/p-1}$,
$\nu_n$ is the surface area of the unit sphere in $\mathbb{R}^n$,
and $C(1,1) = 1$. Then
$$
\int_{\mathbb{R}^n} U(|f(x)|)\, dx<\infty.
$$
\end{theorem}
\begin{proof}
We first consider the case where either $n>1$ or $n=p=1$
and $\lim_{t\to0}t^{-1}\sigma_p(f, t)=\infty$.
Set $\zeta(s) = C(n, p)v_p^*(f, s)$, where $v^*$ is defined by equality (\ref{vstar}).
Under our assumptions, the function $\zeta(\cdot)$ is continuous and strictly increasing.
Set $R= \zeta(N)$.

Consider the functions $u_m$ such that $u_m(x) = U(|f(x)|)$ if $|x|\le m$ and $U(|f(x)|)\le m$
and $u_m(x) = 0$ at all other points $x$.
For these functions we have
\begin{multline*}
\int_{\mathbb{R}^n} u_m(x)\, dx = \int_0^m \lambda^n(u_m\ge t)\, dt
\le
\int_0^m \lambda^n(U(|f|)\ge t)\, dt
\\
= \int_0^{U(r)}\lambda^n(U(|f|)\ge t)\, dt +
\int_{U(r)}^{U(R)}\lambda^n(U(|f|)\ge t)\, dt+\int_{U(R)}^m\lambda^n(U(|f|)\ge t)\, dt.
\end{multline*}
For the first term we have
\begin{multline*}
\int_0^{U(r)}\lambda^n(U(|f|)\ge t)\, dt
=
\int_0^{U(r)}\lambda^n(|f|\ge U^{-1}(t))\, dt
\\
\le \int_0^{U(r)}\lambda^n(|f|\ge (a^{-1}t)^{1/p})\, dt
\le \int_0^\infty\lambda^n(a|f|^p\ge t)dt = a\int_{\mathbb{R}^n}|f(x)|^p\, dx,
\end{multline*}
where in the second inequality we have used that $a(U^{-1}(t))^p\ge t$ if $t\in(0, U(r))$.

For the second term we have
$$
\int_{U(r)}^{U(R)}\lambda^n(U(|f|)\ge t)\, dt\le (U(R) - U(r))\lambda^n(|f|\ge r)
\le
(U(R) - U(r))r^{-p}\int_{\mathbb{R}^n}|f(x)|^p\, dx.
$$
For the third term, by Lemma \ref{lem3.1}, we have
$$
\int_{U(R)}^m\lambda^n(U(|f|)\ge t)\, dt
=
\int_{U(R)}^m\lambda^n\Bigl(|f|\ge \zeta\bigl(\zeta^{-1}(U^{-1}(t))\bigr)\Bigr)\, dt
\le
\int_{U(R)}^m\bigr[\bigl(U\circ \zeta \bigr)^{-1}(t)\bigl]^{-p}\, dt.
$$
The last integral is the area under the graph of
the strictly decreasing function $\bigr[\bigl(U\circ \zeta\bigr)^{-1}(t)\bigl]^{-p}$.
Hence it is equal to
\begin{multline*}
\int\limits_{[(U\circ \zeta)^{-1}(m)]^{-p}}^{[\zeta^{-1}(R)]^{-p}} \bigl(U\circ \zeta\bigr)(t^{-1/p})\, dt
-
U(R)\Bigl([\zeta^{-1}(R)]^{-p} - [(U\circ \zeta)^{-1}(m)]^{-p}\Bigr)
\\
+
\bigl(m - U(R)\bigr)[(U\circ \zeta)^{-1}(m)]^{-p}.
\end{multline*}
The first integral in the above expression is equal to
$$
\int_{N}^{(U\circ \zeta)^{-1}(m)} s^{-1-p}U\bigl(\zeta(s)\bigr)\, ds\le\int_N^{+\infty} s^{-1-p}U\bigl(C(n,p)v_p^*(f, s)\bigr)\, ds<\infty,
$$
since $v_p^*(f, s)=sv_p(f, 1/s)=s\sigma_p(f, s^{-p/n})$.
We also note that
$$
m [(U\circ \zeta)^{-1}(m)]^{-p} \le p\int^{+\infty}_{(U\circ \zeta)^{-1}(m)}t^{-1-p}U\bigl(\zeta(t)\bigr)\, dt.
$$
Summing up these estimates and taking the limit as $m$ tends to infinity, by Fatou's lemma,
we get
\begin{multline}\label{ineq1}
\int_{\mathbb{R}^n} U(|f(x)|)\, dx
\le a\int_{\mathbb{R}^n}|f(x)|^p\, dx + (U(C(n,p)v_p^*(f, N)) - U(r))r^{-p}\int_{\mathbb{R}^n}|f(x)|^p\, dx
\\
+ \int_N^{+\infty} s^{-1-p}U\bigl(C(n,p)s\sigma_p(f, s^{-p/n})\bigr)\, ds.
\end{multline}
which completes the proof in the case under consideration.

In the case where $n=p=1$
and $\lim_{t\to0}t^{-1}\sigma_p(f, t)=A$ for some constant $A$
 (the limit exists by monotonicity)
the function $f$ has bounded variation and is bounded by the constant $A$.
Thus,
$$
\int_{\mathbb{R}} U(|f(x)|)\, dx \le a\int_{\{|f|<r\}}|f(x)|\, dx + U(A)\lambda(|f|\ge r)
\le (a+ r^{-1} U(A))\int_{\mathbb{R}} |f(x)|\, dx.
$$
The theorem is proved.
\end{proof}

\begin{remark}
{\rm
We note that the condition
$$
\int_N^{+\infty} t^{-1-p}U\bigl(C(n,p)t\sigma_p(f, t^{-p/n})\bigr)\, dt <\infty,
$$
is equivalent to the condition
$$
\int_{N'}^{+\infty} s^{-1-n}U\bigl(C(n,p)s^{n/p-1}\sigma^*_p(f, s))\bigr)\, ds<\infty
$$
which can be verified by the change of variables $t=s^{n/p}$.
}
\end{remark}

\begin{remark}
{\rm
Let us consider the case $n=p=1$ in Theorems \ref{T3.2} and \ref{T3.3}.
In this case the condition in the first theorem coincides with
$$
\int_N^{+\infty}\sigma_p\bigl(f, t^{-1}\bigr)\, dU(t) <\infty
$$
and the condition in the second one coincides with
$$
\int_N^{+\infty} t^{-2}U\bigl(t\sigma_p(f, t^{-1})\bigr)\, dt <\infty.
$$
Both conditions are integral forms of the classical
Ulyanov conditions from \cite{UlIzv} and \cite{UlMSb}, formulated in the introduction.
}
\end{remark}

\section{Besov classes on spaces with Gaussian measures}\label{sect4}

We now proceed to the infinite-dimensional Gaussian case.
Let $X$ be a real Hausdorff locally convex space
with  the topological dual space $X^*$.

We recall that a Borel measure $\gamma$ on $X$ is called Radon measure if
for every Borel set $B\subset X$ and every $\varepsilon>0$ there is
a compact set $K\subset B$ such that $\gamma(B\setminus K)<\varepsilon$.
We also recall that a Radon measure $\gamma$ on~$X$
is a centered Gaussian measure if,
for every continuous linear functional $l$
on $X$, the image measure $\gamma\circ l^{-1}$
is either  Dirac's measure at zero or has a density
of the form $(2\pi c^2)^{-1/2}\exp(-t^2/2c^2)$.
From now on let $\gamma$
be a Radon centered Gaussian measure on~$X$.

For a function $f\in L^p(\gamma)$ we set
$$
\|f\|_p := \|f\|_{L^p(\gamma)} := \biggl(\int_{X} |f|^p\, d\gamma\biggr)^{1/p}.
$$

Recall that the Cameron--Martin norm of a vector $h\in X$ is defined by
$$
|h|_H=\sup \biggl\{ l(h)\colon\, \int_X l^2\, d\gamma \le 1, \ l\in X^{*}\biggr\}.
$$
Let $H\subset X$ be the linear subspace of all vectors $h\in X$ such that $|h|_H<\infty$.
This subspace $H$ is called the Cameron--Martin space of the measure $\gamma$.
If $\gamma$ is the standard Gaussian measure on $\mathbb{R}^n$, then its
Cameron--Martin space is $\mathbb{R}^n$ itself and
if $\gamma$ is the countable power of the standard Gaussian measure on the real line,
then $H$ is the classical Hilbert space~$l^2$.
For a general Radon centered Gaussian measure, the Cameron--Martin space
is also a separable Hilbert space (see \cite[Theorem 3.2.7 and Proposition 2.4.6]{GM})
with the inner product $\langle\cdot,\cdot\rangle_H$ generated by the Cameron--Martin norm $|\cdot|_H$.

Let $\{l_i\}_{i=1}^\infty\subset X^*$ be an orthonormal basis in the closure $X_\gamma^*$
of the set $X^*$ in $L^2(\gamma)$.
There is an orthonormal basis $\{e_i\}_{i=1}^\infty$ in $H$ such that $l_i(e_j) = \delta_{i,j}$
(see \cite{GM}).
We will use below that for any orthonormal family $l_1,\ldots,l_n\in X_\gamma^*$
the distribution of the vector $(l_1, \ldots, l_n)$, i.e., the image of the measure~$\gamma$,
is the standard Gaussian measure $\gamma_n$ on $\mathbb{R}^n$, i.e., the measure with density
$(2\pi)^{-n/2}\exp(-|x|^2/2)$ with respect to the standard Lebesgue measure on $\mathbb{R}^n$.

Let $\mathcal{FC}^\infty(X)$ denote the set of all functions
$\varphi$ on $X$ of the form $\varphi(x) = \psi(l_1(x), \ldots, l_n(x))$,
where $\psi\in C_b^\infty(\mathbb{R}^n)$, $l_i\in X^*$, and let
$\mathcal{FC}_0^\infty(X)$ denote the set of all functions
$\varphi$ on $X$ of the form $\varphi(x) = \psi(l_1(x), \ldots, l_n(x))$,
where $\psi\in C_0^\infty(\mathbb{R}^n)$, $l_i\in X^*$.
Let $\mathcal{FC}^\infty(X, H)$ be the set of all vector fields $\Phi$ of the form
$$
\Phi(x) = \sum_{i=1}^n\Psi_i(g_1(x), \ldots, g_n(x))h_i,
$$
where $\Psi_i\in C_b^\infty(\mathbb{R}^n)$, $g_i\in X^*$, $h_i\in H$ and
let $\mathcal{FC}_0^\infty(X, H)$ be the subset of this class consisting of mappings for which
$\Psi_i$ can be chosen with compact support. Note that here we can actually take vectors
$h_i$ orthogonal in~$H$ and functionals $g_i$ orthogonal in~$X_\gamma^*$ such that $g_i(h_j)=\delta_{ij}$.
We will call such vectors and functionals biorthogonal.

For every $\varphi\in\mathcal{FC}^\infty(X)$ of the form
$\varphi(x) = \psi(l_1(x), \ldots, l_n(x))$ set
$$
\nabla\varphi(x) = \sum_{j=1}^n\partial_{x_j}\psi (l_1(x), \ldots, l_n(x)) e_j,
$$
where $\{l_i\}$ and $\{e_i\}$ are biorthogonal.
Let ${\rm div}_\gamma$ be the ``adjoint operator'' to the gradient operator $\nabla$
with respect to~$\gamma$, that is,
$$
\int_X ({\rm div}_\gamma\Phi) \varphi\, d\gamma = -\int_X \langle\Phi, \nabla\varphi\rangle_H\, d\gamma.
$$
for arbitrary $\Phi\in \mathcal{FC}^\infty(X, H)$ and $\phi\in \mathcal{FC}^\infty(X)$.
One can easily check that
$$
{\rm div}_\gamma\Phi (x) = \sum_{j=1}^n\partial_{x_j}\Psi_j(l_1(x), \ldots, l_n(x)) -
l_j(x)\Psi_j(l_1(x), \ldots, l_n(x))
$$
for a vector field $\Phi\in\mathcal{FC}^\infty(X, H)$ of the form
$$
\Phi(x) = \sum_{i=1}^n\Psi_i(l_1(x), \ldots, l_n(x))e_i
$$
with biorthogonal $\{l_i\}$ and $\{e_i\}$.
We note that
for a vector field $\Phi$ from $\mathcal{FC}_0^\infty(X, H)$
its divergence ${\rm div}_\gamma\Phi$ is a bounded function.

We recall that the Ornstein--Uhlenbeck semigroup is defined by the equality
$$
T_tf(x) := \int_{X} f(e^{-t}x+\sqrt{1-e^{-2t}}y)\, \gamma(dy)
$$
for any function $f\in L^1(\gamma)$.

We now fix an orthonormal basis $\{l_n\}\subset X^*$ in $X_\gamma^*$.
For any function $f\in L^1(\gamma)$ let $\mathbb{E}_nf$ be a function on $\mathbb{R}^n$ such that
$$
\int_{\mathbb{R}^n} \psi\mathbb{E}_nf\, d\gamma_n =
\int_X \psi\bigl(l_1(x), \ldots, l_n(x)\bigr) f(x)\, \gamma(dx)
\quad
\forall\, \psi\in C_b^\infty(\mathbb{R}^n),
$$
where $\gamma_n$ is the standard Gaussian measure on $\mathbb{R}^n$.
This equality actually means that the function $\mathbb{E}_nf(l_1,\ldots,l_n)$
is the conditional expectation of $f$ with respect to the
$\sigma$-field generated by functions $l_1,\ldots,l_n$.
By the known property of conditional expectations, for any function $f\in L^p(\gamma)$, we have
$$
\|f - \mathbb{E}_nf(l_1, \ldots, l_n)\|_p\to 0, \quad n\to\infty.
$$

We also introduce the following
functions $C(p)$ and $c_t$ to be used further:
$$
C(p) := \biggl((2\pi)^{-1/2}\int_{\mathbb{R}}|s|^pe^{-\frac{s^2}{2}}\, ds\biggr)^{1/p}\quad \text{and}\quad
c_t := \int_0^t\frac{e^{-\tau}}{\sqrt{1-e^{-2\tau}}}\, d\tau.
$$
We note that
$$
\frac{e^{-\tau}}{\sqrt{1-e^{-2\tau}}}\le (2t)^{-1/2},
$$
$c_t\le (2t)^{1/2}$, and $\lim\limits_{t\to\infty}c_t = \pi/2$.

\vskip .1in

Let us define the Gaussian modulus of continuity $\sigma_{\gamma, p}(f, \cdot)$
which plays the same role as the
 function $\sigma_p(f, \cdot)$ introduced above in case of $\mathbb{R}^n$.

\begin{definition}\label{D4.1} Let $f\in L^p(\gamma)$.
Set
$$
\sigma_{\gamma, p}(f, \varepsilon):=\sup\Bigl\{ \int {\rm div}_\gamma\Phi f d\gamma,\,
\Phi\in \mathcal{FC}_0^\infty(X, H), \|{\rm div}_\gamma\Phi\|_q\le1, \|\Phi\|_q\le\varepsilon\Bigr\}.
$$
\end{definition}

We note that the function $\sigma_{\gamma, p}(f, \cdot)$ is continuous, concave, and nondecreasing on $(0, +\infty)$,
which can be proved similarly to Lemma \ref{lem1.1}.
Thus, by approximation, in the definition of the
quantity $\sigma_{\gamma, p}(f, \varepsilon)$ the supremum can be taken over
all vector fields $\Phi\in \mathcal{FC}^\infty(X, H)$
with $\|{\rm div}_\gamma\Phi\|_q\le1, \|\Phi\|_q\le\varepsilon$.

Using the previous definition we can now introduce Besov classes
on a locally convex space endowed with a Gaussian measure.

\begin{definition}\label{D4.2} Let $\alpha\in(0,1]$, $p\in[1,\infty)$, $\theta\in[1,\infty]$.
We say that a function $f\in L^p(\gamma)$
belongs to the Gaussian Besov space $B_{p, \theta}^\alpha(\gamma)$
if the quantity
$$
V^{p, \theta, \alpha}(f) = \Bigl(\int_0^\infty \bigl[s^{-\alpha}\sigma_{\gamma, p}(f, s)\bigr]^\theta s^{-1}ds\Bigr)^{1/\theta}
$$
is finite.
\end{definition}

We note that in the case $\theta = \infty$ the above definition coincides with the definition
of the Gaussian Nikolskii--Besov class introduced in \cite{BKP}.

We will give an equivalent description of these Gaussian Besov classes
in terms of the following two characteristics.

\begin{definition}\label{D4.3}
For a function $f\in L^p(\gamma)$, $p\in[1,\infty)$,
set
$$
a_{\gamma, p}(f, t):= \Bigl(\iint |f(e^{-t}x + \sqrt{1- e^{-2t}}y) - f(x)|^p\, \gamma(dx)\gamma(dy)\Bigr)^{1/p}
$$
and
$$
A^{p, \theta, \alpha}_\gamma(f):=
\Bigl(\int_0^\infty \bigl[t^{-\alpha/2}a_{\gamma, p}(f, t)\bigr]^\theta t^{-1}dt\Bigr)^{1/\theta}.
$$
\end{definition}

We note that
$$
\|f-T_tf\|_p\le a_{\gamma, p}(f, t).
$$
In a sense, the function $a_{\gamma, p}(f, \cdot)$ can be regarded as a Gaussian replacement
for the finite-dimensional modulus of continuity $\omega_p(f, \cdot)$,
since we cannot directly use shifts $f_h$ of the function $f\in L^p(\gamma)$,
since these shifts can fail to be in $L^p(\gamma)$.

We need the following technical lemma.

\begin{lemma}\label{lem4.1}
Let $\gamma_n$ be the standard Gaussian measure on $\mathbb{R}^n$. Then
for any function \mbox{$f\in L^p(\gamma_n)$}, where $p\in[1, \infty)$, we have
$$
a_{\gamma_n, p}(f, t)\le
2\sigma_{\gamma_n, p}(f, 2^{-1}C(p)c_t).
$$
\end{lemma}

\begin{proof}
For every function $\varphi\in C_0^\infty(\mathbb{R}^{2n})$ we can write
\begin{multline*}
\iint \varphi(x,y)[f(e^{-t}x + \sqrt{1- e^{-2t}}y) - f(x)]\, \gamma_n(dx)\gamma_n(dy)
\\
=
\int f(u) \int [\varphi(e^{-t}u - \sqrt{1- e^{-2t}}v,\sqrt{1- e^{-2t}}u + e^{-t}v) - \varphi(u,v)]\, \gamma_n(dv)\, \gamma_n(du)
\\
=
\int f(u) \int_0^t \frac{\partial}{\partial s} g_s(u)\, ds\, \gamma_n(du),
\end{multline*}
where
$$
g_s(u) := \int\varphi(e^{-s}u - \sqrt{1- e^{-2s}}v,\sqrt{1- e^{-2s}}u + e^{-s}v)\, \gamma_n(dv).
$$
We now note that for an arbitrary function $\psi\in C_0^\infty(\mathbb{R}^n)$ we have
\begin{multline*}
\int\psi(u) \frac{\partial}{\partial s} g_s(u)\, \gamma_n(du)
\\
=
\frac{\partial}{\partial s} \int \psi(u) \int\varphi(e^{-s}u - \sqrt{1- e^{-2s}}v,\sqrt{1- e^{-2s}}u + e^{-s}v)\, \gamma_n(dv)\, \gamma_n(du)
\\
=
\frac{\partial}{\partial s} \iint \varphi(x,y)\psi(e^{-s}x + \sqrt{1- e^{-2s}}y)\, \gamma_n(dx)\, \gamma_n(dy)
\\
=
\frac{e^{-s}}{\sqrt{1- e^{-2s}}}\iint \varphi(x,y) \langle\nabla\psi(e^{-s}x + \sqrt{1- e^{-2s}}y), e^{-s}y - \sqrt{1-e^{-2s}x}\rangle\, \gamma_n(dx)\, \gamma_n(dy)
\\
=
\frac{e^{-s}}{\sqrt{1- e^{-2s}}}\int \Bigl\langle \nabla\psi(u), \int v\varphi(e^{-s}u - \sqrt{1- e^{-2s}}v,\sqrt{1- e^{-2s}}u + e^{-s}v)\, \gamma_n(dv)\Bigr\rangle\, \gamma_n(du)
\\=
-\int\psi(u) {\rm div}_{\gamma_n} G_s(u)\, \gamma_n(du),
\end{multline*}
where
$$
G_s(u) := \frac{e^{-s}}{\sqrt{1- e^{-2s}}}\int v\varphi(e^{-s}u - \sqrt{1- e^{-2s}}v,\sqrt{1- e^{-2s}}u + e^{-s}v)\, \gamma_n(dv)
\in C_b^\infty(\mathbb{R}^n).
$$
Thus,
$$
\frac{\partial}{\partial s} g_s(u) = {\rm div}_\gamma \bigl(-G_s(u)\bigr)
$$
and
$$
\int f \int_0^t \frac{\partial}{\partial s} g_s\, ds\, d\gamma_n
=
\int{\rm div}_\gamma \Bigl(-\int_0^t G_s\, ds\Bigr) f d\gamma_n.
$$
We observe that
\begin{multline*}
{\rm div}_{\gamma_n} \Bigl(-\int_0^t G_s\, ds\Bigr)
=
\int_0^t {\rm div}_{\gamma_n} \bigl(-G_s\bigr)\, ds
=
\int_0^t \frac{\partial}{\partial s} g_s(u)\, ds
\\
=
\int [\varphi(e^{-t}u - \sqrt{1- e^{-2t}}v,\sqrt{1- e^{-2t}}u + e^{-t}v) - \varphi(u,v)]\, \gamma_n(dv)
\end{multline*}
and that
$$
\Bigl\|{\rm div}_{\gamma_n} \Bigl(-\int_0^t G_s, ds\Bigr)\Bigr\|_q
\le
2\|\varphi\|_{L^q(\gamma_n\otimes\gamma_n)}.
$$
Moreover, we have
$$
\Bigl\|\int_0^t G_s\, ds\Bigr\|_q
\le
\int_0^t\bigl\| G_s\bigr\|_q\, ds
$$
and it remains to estimate $\bigl\| G_s\bigr\|_q$.
To do this, we note that for an arbitrary vector field $\Psi\in C_0^\infty (\mathbb{R}^n)$
\begin{multline*}
\int \langle\Psi, G_s\rangle \, d\gamma_n
=
\frac{e^{-s}}{\sqrt{1- e^{-2s}}}  \iint \langle\Psi (u), v\rangle \varphi(e^{-s}u - \sqrt{1- e^{-2s}}v,\sqrt{1- e^{-2s}}u + e^{-s}v)\, \gamma_n(dv)\gamma_n(du)
\\
\le
\frac{e^{-s}}{\sqrt{1- e^{-2s}}} \|\varphi\|_{L^q(\gamma_n\otimes\gamma_n)}\Bigl(\iint|\langle\Psi (u), v\rangle|^p\, \gamma_n(dv)\gamma_n(du)\Bigr)^{1/p}
\\
\le
\frac{e^{-s}}{\sqrt{1- e^{-2s}}} C(p)\|\varphi\|_{L^q(\gamma_n\otimes\gamma_n)}\|\Psi\|_p.
\end{multline*}
Thus,
$$
\bigl\| G_s\bigr\|_q
\le
C(p)\frac{e^{-s}}{\sqrt{1- e^{-2s}}}\|\varphi\|_{L^q(\gamma_n\otimes\gamma_n)}
$$
and
$$
\Bigl\|\int_0^t G_s\, ds\Bigr\|_q
\le
\int_0^t\bigl\| G_s\bigr\|_q\, ds
\le
C(p)c_t\|\varphi\|_{L^q(\gamma_n\otimes\gamma_n)}.
$$
Hence, for an arbitrary function $\varphi\in C_0^\infty(\mathbb{R}^{2n})$ with $\|\varphi\|_{L^q(\gamma_n\otimes\gamma_n)}=1$
we have
\begin{multline*}
\iint \varphi(x,y)[f(e^{-t}x + \sqrt{1- e^{-2t}}y) - f(x)]\, \gamma_n(dx)\gamma_n(dy)
=
\int f \int_0^t \frac{\partial}{\partial s} g_s\, ds\, d\gamma_n
\\
=
\int{\rm div}_{\gamma_n} \Bigl(-\int_0^t G_s\, ds\Bigr) f d\gamma_n
\le
2\sigma_{\gamma_n, p}(f, 2^{-1}C(p)c_t).
\end{multline*}
Taking the supremum over functions $\varphi\in C_0^\infty(\mathbb{R}^{2n})$ with
$\|\varphi\|_{L^q(\gamma_n\otimes\gamma_n)}=1$ we obtain the announced bound.
\end{proof}

The following theorem is a Gaussian analog of Theorem \ref{T3.1}.

\begin{theorem}\label{lem4.2}
For any function $f\in L^p(\gamma)$, where $p\in[1, \infty)$, we have
$$
a_{\gamma, p}(f, t)\le
2\sigma_{\gamma, p}(f, 2^{-1}C(p)c_t).
$$
If $p>1$ we have the inverse bound:
$$
\sigma_{\gamma, p}(f, \varepsilon) \le \bigl(1+ C(p/(p-1))\bigr)a_{\gamma, p}(f, \varepsilon^2).
$$
\end{theorem}

\begin{proof}
To prove the first part of the lemma, we
fix an orthonormal basis $\{l_n\}\subset X^*$ in $X_\gamma^*$.
By the previous lemma we have
$$
a_{\gamma_n, p}(\mathbb{E}_nf, t)\le
2\sigma_{\gamma_n, p}(\mathbb{E}_nf, 2^{-1}C(p)c_t)
\le
2\sigma_{\gamma, p}(f, 2^{-1}C(p)c_t).
$$
We observe that $a_{\gamma_n, p}(f, t)\to a_{\gamma, p}(f, t)$ as $n$ tends to infinity,
which completes the proof.

Let now $f\in L^p(\gamma)$ for some $p>1$.
For an arbitrary vector field $\Phi\in \mathcal{FC}_0^\infty(X, H)$ we can write
\begin{multline*}
\int {\rm div}_\gamma\Phi T_tf\, d\gamma
=
e^{-t}\int {\rm div}_\gamma T_t\Phi f\, d\gamma
\\
=
-\frac{e^{-t}}{\sqrt{1-e^{-2t}}}
\iint f(u)\langle\Phi(e^{-t}u-\sqrt{1-e^{-2t}}v), e^{-t}v + \sqrt{1-e^{-2t}}u\rangle_H\, \gamma(dv)\gamma(du)
\\
=
-\frac{e^{-t}}{\sqrt{1-e^{-2t}}}
\iint f(e^{-t}x + \sqrt{1-e^{-2t}}y)\langle\Phi(x), y\rangle_H\, \gamma(dy)\gamma(dx).
\end{multline*}
We observe that
$$
\int f(x)\langle\Phi(x), y\rangle_H\, \gamma(dy) = 0
$$
for an arbitrary fixed point $x$.
Thus, the last expression is equal to
\begin{multline*}
-\frac{e^{-t}}{\sqrt{1-e^{-2t}}}
\iint \langle\Phi(x), y\rangle_H \bigl[f(e^{-t}x + \sqrt{1-e^{-2t}}y) - f(x)\bigr]\, \gamma(dy)\gamma(dx)
\\
\le
t^{-1/2}a_{\gamma, p}(f, t)\Bigl(\iint |\langle\Phi(x), y\rangle_H|^q \, \gamma(dy)\gamma(dx)\Bigr)^{1/q}
=
C(q)t^{-1/2}a_{\gamma, p}(f, t) \|\Phi\|_q.
\end{multline*}
So, we have proved the estimate
$$
\int {\rm div}_\gamma\Phi T_tf\, d\gamma
\le
C(q)t^{-1/2}a_{\gamma, p}(f, t)\|\Phi\|_q.
$$
Now we have
$$
\int {\rm div}_\gamma\Phi f\, d\gamma
=
\int {\rm div}_\gamma\Phi [f - T_tf]\, d\gamma
+
\int {\rm div}_\gamma\Phi T_tf\, d\gamma.
$$
The first term in the above expression is estimated by
$$
a_{\gamma, p}(f, t)\|{\rm div}_\gamma\Phi\|_q
$$
and the second term, as we have proved, is not greater than
$$
C(q)t^{-1/2}a_{\gamma, p}(f, t)\|\Phi\|_q.
$$
Taking $t = \varepsilon^2$ we obtain
$$
\sigma_{\gamma, p}(f, \varepsilon) \le (1+ C(q))a_{\gamma, p}(f, \varepsilon^2),
$$
which is the announced bound.
\end{proof}

\begin{corollary}\label{c4.1}
For any function $f\in B^\alpha_{p, \theta}(\gamma)$, where $p\in[1, \infty)$,
and for any Lipschitz function $u\colon\mathbb{R}\to\mathbb{R}$ we have
$$
a_{\gamma, p}(u(f), t)\le  2^{1-\alpha}(\alpha\theta)^{1/\theta}Lip(u)C(p)^\alpha c_t^\alpha V^{p, \theta, \alpha}(f),
$$
where $Lip(u)$ is the Lipschitz constant of the function $u$.
\end{corollary}

\begin{proof}
By the Lipschitz continuity of the function $u$ and by the previous lemma we can write
\begin{multline*}
a_{\gamma, p}(u(f), t)
=
\biggl(\iint\Bigl|u\bigl(f(e^{-t}x + \sqrt{1- e^{-2t}y})\bigr) - u\bigl(f(x)\bigr)\Bigr|^p\, \gamma(dy)\gamma(dx)\biggr)^{1/p}
\\
\le
Lip(u)\biggl(\iint\Bigl|f(e^{-t}x + \sqrt{1- e^{-2t}y}) - f(x)\Bigr|^p\, \gamma(dy)\gamma(dx)\biggr)^{1/p}
\\
=
Lip(u)a_{\gamma, p}(f, t)
\le
2Lip(u)\sigma_{\gamma, p}(f, 2^{-1}C(p)c_t).
\end{multline*}
We now note that
$$
\bigl(2^{-1}C(p)c_t\bigr)^{-\alpha\theta}\bigl[\sigma_{\gamma, p}(f, 2^{-1}C(p)c_t)\bigr]^\theta
\le
\alpha\theta\int\limits_{2^{-1}C(p)c_t}^\infty r^{-\alpha\theta-1} [\sigma_{\gamma, p}(f, r)]^\theta dr
\le
\alpha\theta \bigl[V^{p, \theta, \alpha}(f)\bigr]^\theta.
$$
Thus,
$$
a_{\gamma, p}(u(f), t)
\le
2^{1-\alpha}(\alpha\theta)^{1/\theta}Lip(u)C(p)^\alpha c_t^\alpha V^{p, \theta, \alpha}(f)
$$
as announced.
\end{proof}

Also, as a corollary, we obtain that
the conditions $V^{p, \theta, \alpha}_\gamma(f)<\infty$ and $A^{p, \theta, \alpha}_\gamma(f)<\infty$
are equivalent for $p>1$.

\begin{corollary}\label{t4.1}
For any function $f\in B^\alpha_{p, \theta}(\gamma)$, where $p\in[1, \infty)$, we have
$$
A^{p, \theta, \alpha}_\gamma(f)\le
2^{1-\alpha+1/\theta}C(p)^{\alpha} V^{p, \theta, \alpha}_\gamma(f).
$$
Moreover, for $p\in(1,\infty)$ we have the inverse statement, that is,
if for a function $f\in L^p(\gamma)$
the quantity $A^{p, \theta, \alpha}_\gamma(f)$ is finite,
 then $f\in B^\alpha_{p, \theta}(\gamma)$
and
$$
V^{p, \theta, \alpha}_\gamma(f)\le 2^{-1/\theta}\bigl(1+ C(p/(p-1))\bigr)A^{p, \theta, \alpha}_\gamma(f).
$$
\end{corollary}

\begin{proof}
For a function $f\in B^\alpha_{p, \theta}(\gamma)$,
by Lemma \ref{lem4.2}, we have
$$
a_{\gamma, p}(f, t)
\le
2\sigma_{\gamma, p}(f, 2^{-1}C(p)c_t)
\le
2\sigma_{\gamma, p}(f, 2^{-1}C(p)t^{1/2}).
$$
Thus,
\begin{multline*}
\bigl[A^{p, \theta, \alpha}_\gamma(f)\bigr]^\theta=\int_0^\infty \bigl[t^{-\alpha/2}a_{\gamma, p}(f, t)\bigr]^\theta t^{-1}dt
\le
2^\theta\int_0^\infty t^{-\alpha\theta/2}\bigl[\sigma_{\gamma, p}(f, 2^{-1}C(p)t^{1/2})\bigr]^\theta t^{-1}dt
\\
=
2^{1+\theta-\alpha\theta}C(p)^{\alpha\theta}\int_0^\infty r^{-\alpha\theta}\bigl[\sigma_{\gamma, p}(f, r)\bigr]^\theta r^{-1}dr
=
2^{1+\theta-\alpha\theta}C(p)^{\alpha\theta}\bigl[V^{p, \theta, \alpha}_\gamma(f)\bigr]^\theta,
\end{multline*}
which is the announced bound.

Conversely, for any function $f\in L^p(\gamma)$ with $p>1$ and finite
$A^{p, \theta, \alpha}_\gamma(f)$,
Lemma \ref{lem4.2} gives that
$$
\sigma_{\gamma, p}(f, \varepsilon) \le \bigl(1+ C(q)\bigr)a_{\gamma, p}(f, \varepsilon^2),
$$
which yields
\begin{multline*}
\bigl[V^{p, \theta, \alpha}_\gamma(f)\bigr]^\theta
=
\int_0^\infty\bigl[r^{-\alpha}\sigma_{\gamma, p}(f, r)\bigr]^\theta r^{-1}dr
\\
\le
\bigl(1 + C(q)\bigr)^\theta\int_0^\infty\bigl[r^{-\alpha}a_{\gamma, p}(f, r^2)\bigr]^\theta r^{-1}dt
=
2^{-1}\bigl(1 + C(q)\bigr)^\theta\int_0^\infty \bigl[t^{-\alpha/2}a_{\gamma, p}(f, t)\bigr]^\theta t^{-1}dt
\\
=
2^{-1}\bigl(1 + C(q)\bigr)^\theta\bigl[A^{p, \theta, \alpha}_\gamma(f)\bigr]^\theta.
\end{multline*}
This is the announced estimate.
\end{proof}

We now proceed to
a log-Sobolev-type embedding theorem for Besov classes with respect to a Gaussian measure.
As we have mentioned in the introduction,
the main idea of the proof is to use the
short time behavior of the Ornstein-Uhlenbeck semigroup
together with its hypercontractivity property, similarly in a sense
to the approach from \cite{Led}.

\begin{theorem}\label{t4.2}
For any function $f\in B^\alpha_{p, \theta}(\gamma)$, where $p\in(1, \infty)$, and for any number $\beta\in (0,\alpha)$
the function $|f||\ln|f||^{\beta/2}$ belongs to $L^p(\gamma)$. Moreover, there is a constant $C=C(p, \theta, \alpha, \beta)$,
depending only on parameters $p$, $\theta$, $\alpha$, and $\beta$, such that
$$
\Bigl(\int|f|^p\bigl|\ln(|f|\|f\|_p^{-1})\bigr|^{p\beta/2} \, d\gamma\Bigr)^{1/p}
\le
C\bigl(\|f\|_p + V_\gamma^{p, \theta, \alpha}(f)\bigr).
$$
\end{theorem}

\begin{proof}
We recall
the hypercontractivity property of the Ornstein--Uhlenbeck semigroup
(see \cite[Theorem 5.5.3]{GM}): for any function $f\in L^p(\gamma)$ one has
$$
\|T_tf\|_{1+(p-1)e^{2t}}\le \|f\|_p.
$$
For an arbitrary number $s>0$, let $A_s:=\{|f|\ge s\}$.
We note that the function $\tau\mapsto \max\{|\tau|, s\}$ is
$1$-Lipschitz.
Thus, for an arbitrary function $\varphi\in \mathcal{FC}^\infty(X)$ and
any number $t>0$,
by Corollary~\ref{c4.1} and by the hypercontractivity property, we have
\begin{multline*}
\int \varphi I_{A_s}(|f| - s)\, d\gamma
=
\int I_{A_s}\varphi (\max\{|f|, s\} - s)\, d\gamma
\\
=
\int I_{A_s}\varphi \bigl[\max\{|f|, s\}-T_t(\max\{|f|, s\})\bigr]\, d\gamma + \int I_{A_s}\varphi T_t(\max\{|f|, s\} - s)\, d\gamma
\\
\le
\|\varphi\|_q\|\max\{|f|, s\} - T_t(\max\{|f|, s\})\|_p + \|I_{A_s}\varphi\|_{\frac{1+(p-1)e^{2t}}{(p-1)e^{2t}}}\|T_t(\max\{|f|, s\} - s)\|_{1+(p-1)e^{2t}}
\\
\le
2^{1-\alpha}(\alpha\theta)^{1/\theta}C(p)^\alpha\|\varphi\|_q V_\gamma^{p, \theta, \alpha}(f)t^{\alpha/2}
+
\|I_{A_s}\varphi\|_{\frac{1+(p-1)e^{2t}}{(p-1)e^{2t}}}\|I_{A_s}(|f| - s)\|_p.
\end{multline*}
We note that
$$
\frac{1+(p-1)e^{2t}}{(p-1)e^{2t}} = 1 + \frac{1}{(p-1)e^{2t}} = q (1/q + 1/(pe^{2t}))\le q.
$$
Thus, we can apply H\"older's inequality to
the expression $\|I_{A_s}\varphi\|_{\frac{1+(p-1)e^{2t}}{(p-1)e^{2t}}}$
with the exponents $(1/p - 1/(pe^{2t}))^{-1}$
and $(1/q + 1/(pe^{2t}))^{-1}$, which yields
$$
\|I_{A_s}\varphi\|_{\frac{1+(p-1)e^{2t}}{(p-1)e^{2t}}}\le [\gamma(A_s)]^{\frac{e^{2t}-1}{q+pe^{2t}}}\|\varphi\|_q.
$$
Taking the supremum over functions $\varphi$ with $\|\varphi\|_q=1$ we obtain the estimate
$$
\|I_{A_s}(|f| - s)\|_p\le
2^{1-\alpha}(\alpha\theta)^{1/\theta}C(p)^\alpha V_\gamma^{p, \theta, \alpha}(f)t^{\alpha/2}
+
[\gamma(A_s)]^{\frac{e^{2t}-1}{q+pe^{2t}}}\|I_{A_s}(|f| - s)\|_p.
$$
We now observe that
$$
\frac{e^{2t}-1}{q+pe^{2t}}
=
p^{-1}\frac{q^{-1}(e^{2t}-1)}{1+q^{-1}(e^{2t}-1)}
\ge
p^{-1}\frac{2q^{-1}t}{1+2q^{-1}t}
\ge
\frac{t}{pq}
$$
whenever $t\le 1/2$. Thus, whenever $t\le 1/2$, we have
$$
\|I_{A_s}(|f| - s)\|_p\le
2^{1-\alpha}(\alpha\theta)^{1/\theta}C(p)^\alpha V_\gamma^{p, \theta, \alpha}(f)t^{\alpha/2}
+
[\gamma(A_s)]^{(pq)^{-1}t}\|I_{A_s}(|f| - s)\|_p.
$$
For the sets $A_s$ with $\gamma(A_s)\le e^{-2pq}$
we can take $t= pq(-\ln\gamma(A_s))^{-1}\le1/2$
and conclude that
$$
\|I_{A_s}(|f| - s)\|_p\le
2^{1-\alpha}(\alpha\theta)^{1/\theta}C(p)^\alpha(pq)^{\alpha/2} V_\gamma^{p, \theta, \alpha}(f)[-\ln\gamma(A_s)]^{-\alpha/2}
+
e^{-1}\|I_{A_s}(|f| - s)\|_p,
$$
since
$$
[\gamma(A_s)]^{(pq)^{-1}t} = e^{- (pq)^{-1}[-\ln\gamma(A_s)]t}=e^{-1}
$$
for such $t$.
The obtained inequality can be rewritten in the form
$$
\|I_{A_s}(|f| - s)\|_p\le C(p, \theta, \alpha)V_\gamma^{p, \theta, \alpha}(f)[-\ln\gamma(A_s)]^{-\alpha/2},
$$
where $C(p, \theta, \alpha)=2^{1-\alpha}(\alpha\theta)^{1/\theta}C(p)^\alpha e(e-1)^{-1}(pq)^{\alpha/2}$.
We now observe that $\gamma(A_s)\le \|f\|_p^p s^{-p}$ and
$I_{A_s}(|f| - s)\ge 2^{-1}I_{A_{2s}}|f|$.
Thus, if $t\ge e^{2q}$, taking $s= t\|f\|_p$, we have
$$
\int I_{\{|f|\ge2t\|f\|_p\}}|f|^p\, d\gamma \le \bigl(2p^{-\alpha/2}C(p, \theta, \alpha)\bigr)^p\bigl[V_\gamma^{p, \theta, \alpha}(f)\bigr]^p[\ln t]^{-p\alpha/2}.
$$
Multiplying both sides of the inequality by $t^{-1} [\ln t]^{-1+p\beta/2}$
and integrating with respect to $t$ from $e^{2q}$ to $+\infty$ we obtain
\begin{multline*}
\int_{e^{2q}}^\infty  t^{-1} [\ln t]^{-1+p\beta/2}\int I_{\{|f|\ge 2t\|f\|_p\}}|f|^p\, d\gamma\, dt
\\
\le
\bigl(2p^{-\alpha/2}C(p, \theta, \alpha)\bigr)^p\bigl[V_\gamma^{p, \theta, \alpha}(f)\bigr]^p\int_{e^{2q}}^\infty[\ln t]^{-1-p(\alpha-\beta)/2}t^{-1}dt
\\
=
\bigl(2p^{-\alpha/2}C(p, \theta, \alpha)\bigr)^p2p^{-1}(\alpha - \beta)^{-1}(2q)^{-p(\alpha-\beta)/2}\bigl[V_\gamma^{p,\alpha}(f)\bigr]^p.
\end{multline*}
The left-hand side of the above estimate is equal to
\begin{multline*}
\int|f|^p I_{\{|f|\ge 2e^{2q}\|f\|_p\}}\int_{e^{2q}}^{|f|\|f\|_p^{-1}2^{-1}} t^{-1} [\ln t]^{-1+p\beta/2}\, dt\, d\gamma
\\
=
2(p\beta)^{-1}\int|f|^p I_{\{|f|\ge 2e^{2q}\|f\|_p\}}\bigl([\ln(|f|\|f\|_p^{-1}2^{-1})]^{p\beta/2} - (2q)^{p\beta/2}\bigr)\, d\gamma
\\
\ge
2(p\beta)^{-1}p^{-p\beta/2}\int|f|^p I_{\{|f|\ge 2e^{2q}\|f\|_p\}}[\ln(|f|\|f\|_p^{-1})]^{p\beta/2}\, d\gamma
- 2(p\beta)^{-1}(2q)^{p\beta/2}\|f\|_p^p
\\
=
2(p\beta)^{-1}p^{-p\beta/2}\int|f|^p\bigl|\ln(|f|\|f\|_p^{-1})\bigr|^{p\beta/2}
- 2(p\beta)^{-1}p^{-p\beta/2}\int|f|^p I_{\{|f|< 2e^{2q}\|f\|_p\}}\bigl|\ln(|f|\|f\|_p^{-1})\bigr|^{p\beta/2}\, d\gamma
\\
- 2(p\beta)^{-1}(2q)^{p\beta/2}\|f\|_p^p
\end{multline*}
Thus, since $a|\ln a|^{\beta/2}\le 2e^{2q}(2q+1)^{\beta/2}$ if $a\in[0, 2e^{2q}]$, we have
$$
\int|f|^p\bigl|\ln(|f|\|f\|_p^{-1})\bigr|^{p\beta/2} \, d\gamma
\\
\le
C_1(p, \theta, \alpha, \beta)\bigl[V_\gamma^{p,\alpha}(f)\bigr]^p
+
C_2(p, \alpha, \theta, \beta)\|f\|_p^p
$$
with
$$
C_1(p,\theta, \alpha, \beta) = \bigl(2p^{-\alpha/2}C(p, \theta, \alpha)\bigr)^p(\alpha - \beta)^{-1}(2q)^{-p(\alpha-\beta)/2}p^{p\beta/2}\beta
$$
and
$$
C_2(p, \theta, \alpha, \beta) =(2^pe^{2qp}(2q+1)^{p\beta/2} + (2q)^{p\beta/2}p^{p\beta/2})
$$
It is readily seen that the obtained bound is equivalent to the announced assertion.
The theorem is proved.
\end{proof}

Finally,
let us discuss estimates for the best approximations by Hermite polynomials in $L^2(\gamma)$
with respect to a Gaussian measure $\gamma$.
Recall (see \cite[Section 2.9]{GM}) that the space $L^2(\gamma)$ can be decomposed into the direct sum of mutually orthogonal subspaces $\mathcal{H}_k$
consisting of the so-called Hermite polynomials of a fixed degree $k$:
$$
L^2(\gamma) = \bigoplus_{k=0}^\infty \mathcal{H}_k.
$$
This decomposition is also called the Wiener chaos decomposition.
The space $\mathcal{H}_k$ is actually the orthogonal complement of the space of all measurable
polynomials of degree $k-1$ in the space of all measurable polynomials of degree $k$.
Let $I_k$ be the projection operator to the subspace $\mathcal{H}_k$.
For any function $f\in L^2(\gamma)$ set
$$
E_N(f):= \inf\{\|f - f_N\|_2;\, f_N\in \bigoplus_{k=0}^{N-1} \mathcal{H}_k\}.
$$
The quantity $E_N(f)$ is the value of the best approximation of the function $f$
by linear combinations of Hermite polynomials of the given degree.
It is clear that
$$
E_N(f) = \|f - I_0(f) - \ldots - I_{N-1}(f)\|_2.
$$

We now prove a Jackson--Stechkin-type inequality for the quantity $E_N(f)$
involving the Gaussian modulus of continuity $\sigma_{\gamma, 2}(f, \cdot)$.

\begin{theorem} For any function $f\in L^2(\gamma)$ we have
$$
E_{N-1}(f)\le \sigma_{\gamma,2}(f, \sqrt{2\pi} N^{-1/2}).
$$
\end{theorem}

\begin{proof}
For an arbitrary function $\varphi\in\mathcal{FC}^\infty(X)$ with $\|\varphi\|_2\le1$
we have
\begin{multline*}
\int \varphi (f - I_0(f) - \ldots - I_{N-1}(f))\, d\gamma
=
\int (\varphi - I_0(\varphi) - \ldots - I_{N-1}(\varphi)) f \, d\gamma
\\
=
\int {\rm div}_\gamma\Bigl(-\int_0^\infty \nabla T_t(\varphi - I_0(\varphi) - \ldots - I_{N-1}(\varphi))\, dt\Bigr) f\, d\gamma,
\end{multline*}
where we have used the equality
$$
\psi(x)= - \int_0^\infty L T_t\psi(x)\, dt =
{\rm div}_\gamma\Bigl(-\int_0^\infty \nabla T_t\psi(x)\, dt\Bigr)
$$
for an arbitrary function $\psi\in \mathcal{FC}^\infty(X)$ with $\int\psi\, d\gamma=0$,
where $L$ is the Ornstein--Uhlenbeck operator (see \cite[Section 1.4 and Remark 5.8.7]{GM}).
Recall that
$$
\|\nabla T_t\psi\|_2\le \frac{e^{-t}}{\sqrt{1-e^{-2t}}}\|\psi\|_2
$$
for any function $\psi\in\mathcal{FC}^\infty(X)$
and that
$$
T_tg = \sum_{k=0}^\infty e^{-kt}I_k(g),
$$
which yields the estimate
$$
\|T_t(g - I_0(g) - \ldots - I_{N-1}(g))\|_2\le e^{-Nt}\|g\|_2,
$$
for all $g\in L^2(\gamma)$.
We now note that
\begin{multline*}
\Bigl\|\int_0^\infty \nabla T_t(\varphi - I_0(\varphi) - \ldots - I_{N-1}(\varphi))\, dt\Bigr\|_2
\le
\int_0^\infty\|\nabla T_t(\varphi - I_0(\varphi) - \ldots - I_{N-1}(\varphi))\|_2\, dt
\\
\le
\int_0^\infty\frac{e^{-t/2}}{\sqrt{1-e^{-t}}}\|T_{t/2}(\varphi - I_0(\varphi) - \ldots - I_{N-1}(\varphi))\|_2\, dt
\\
\le
\|\varphi\|_2\int_0^\infty\frac{e^{-t/2}}{\sqrt{1-e^{-t}}}e^{-Nt/2}\, dt
=
B((N+1)/2, 1/2) \|\varphi\|_2,
\end{multline*}
where $B(x, y)$ is the standard beta function.
It can be easily verified that
$$
\frac{\Gamma(x+1/2)}{\Gamma(x)}\le\sqrt{x},
$$
where $\Gamma(\cdot)$ is the standard gamma function.
Indeed, introducing the probability density
$$
\rho_x(t):=[\Gamma(x)]^{-1}t^{x-1}e^{-t}I_{\{t>0\}}
$$
and applying Jensen's inequality we obtain
$$
\frac{\Gamma(x+1/2)}{\Gamma(x)} = \int \sqrt t \rho_x(t)\, dt\le \sqrt{\int t\rho_x(t)\, dt}
=
\sqrt{\frac{\Gamma(x+1)}{\Gamma(x)}} = \sqrt{x}.
$$
Thus,
$$
B((N+1)/2, 1/2) = \frac{\Gamma((N+1)/2)\Gamma(1/2)}{\Gamma(1+N/2)}
=
\frac{\sqrt{\pi}\Gamma(N/2+1/2)}{N/2 \Gamma(N/2)}\le
\sqrt{2\pi}N^{-1/2}.
$$
Therefore,
$$
\Bigl\|\int_0^\infty \nabla T_t(\varphi - I_0(\varphi) - \ldots - I_{N-1}(\varphi))\, dt\Bigr\|_2\le
\sqrt{2\pi}N^{-1/2}.
$$
We also note that
$$
\Bigl\|{\rm div}_\gamma\Bigl(-\int_0^\infty \nabla T_t(\varphi - I_0(\varphi) - \ldots - I_{N-1}(\varphi))\, dt\Bigr)\Bigr\|_2
=
\|\varphi - I_0(\varphi) - \ldots - I_{N-1}(\varphi)\|_2\le \|\varphi\|_2\le1.
$$
So,
$$
\int \varphi (f - I_0(f) - \ldots - I_{N-1}(f))\, d\gamma\le \sigma_{\gamma, 2}(f, \sqrt{2\pi}N^{-1/2}),
$$
which completes  the proof.
\end{proof}

\end{document}